\documentclass[11pt,reqno]{amsart}
\usepackage[T1]{fontenc}
\usepackage{lmodern}
\usepackage{amsmath,amssymb,amsthm,mathtools}
\usepackage{microtype}
\usepackage[a4paper,margin=28mm]{geometry}
\usepackage{booktabs}
\usepackage{needspace}
\usepackage[colorlinks=true,linkcolor=blue!45!black,citecolor=blue!45!black,urlcolor=blue!45!black]{hyperref}
\usepackage{xcolor}
\hypersetup{pdftitle={Noninteger completely monotonic degrees of Stirling remainders and their derivatives},pdfauthor={Valmir Krasniqi}}
\newtheorem{theorem}{Theorem}[section]
\newtheorem{proposition}[theorem]{Proposition}
\newtheorem{lemma}[theorem]{Lemma}
\newtheorem{corollary}[theorem]{Corollary}
\newtheorem{conjecture}{Conjecture}

\theoremstyle{remark}

\newcommand{\CM}{\mathrm{CM}}
\newcommand{\dcm}{d_{\CM}}
\newcommand{\Lap}{\mathcal L}
\newcommand{\dd}{\,dt}
\numberwithin{equation}{section}
\allowdisplaybreaks[1]
\title[Noninteger degrees of Stirling remainders]{Noninteger completely monotonic degrees of Stirling remainders and their derivatives}
\author{Valmir Krasniqi}
\address{Institute of Science, Technology, Engineering and Mathematics (STEM), Rr. Nekibe Kelmendi Nr. 26, 10000 Prishtin\"e, Republic of Kosovo}
\email{valmir@stem-ks.org}
\date{October 22, 2026}
\subjclass[2020]{Primary 33B15, 26A48; Secondary 26A33, 44A10}
\keywords{Completely monotonic degree, Stirling remainder, gamma function, digamma function, fractional derivative, Laplace transform}

\begin{document}
\begin{abstract}
For Stirling expansions truncated after at least two Bernoulli terms, we disprove the three real-degree conjectures recorded by Mahmoud and Qi for the signed remainders of the logarithm of the gamma function and their derivatives.
We give a unified proof of the corresponding integer-degree formula, recovering known cases and resolving Zhu's conjectures for higher derivatives in this range.
Every real completely monotonic degree lies strictly between the predicted integer and the next integer.
We obtain a universal fractional lower bound through an explicit inverse function and a simpler rational bound valid for every derivative order.
A local estimate for the negative part of the Laplace kernel's derivative with respect to its logarithmic variable makes the lower improvement proportional to the reciprocal of the combined truncation and derivative order.
Together with a uniform upper bound, this determines the exact order of decay of the fractional excess for each fixed derivative order and for its supremum over all derivative orders.
Sharper kernel estimates give stronger bounds for the first two degrees.
The proof combines fractional positivity criteria, the classical Laplace representation, a sign characterization of the associated densities, and a convexity obstruction at the upper integer endpoint.
All estimates used in the proofs are analytic.
The exact real degrees and the optimal asymptotic constants remain open.
\end{abstract}
\maketitle

\section{Introduction and main results}

The distinction between real and integer exponents is decisive for the completely monotonic degrees of Stirling remainders.
We show that the conjectured integer values are their exact integer parts, while every real degree in the truncation range considered here is strictly larger.
We also determine the order of decay of this fractional excess as the truncation order increases.

\subsection{Notation and setting}

A function $f\in C^\infty(0,\infty)$ is completely monotonic if
\[
(-1)^j f^{(j)}(x)\ge0
\qquad(x>0,\ j=0,1,2,\ldots).
\]
We write $f\in\CM$ for this property.
Miller and Samko \cite{MillerSamko2001} survey its basic properties and examples from special functions.
For such a function with $f(x)\to0$ as $x\to\infty$, define its real completely monotonic degree by
\begin{equation}\label{eq:degree}
\dcm(f)=\sup\{\rho\ge0:x^\rho f(x)\in\CM\}.
\end{equation}
The supremum is taken over real exponents, a distinction central to this paper.
The notion of completely monotonic degree was introduced by Guo and Qi \cite{GuoQi2012}, who determined an exact degree for a function involving the trigamma function.

Let the Bernoulli numbers be determined by
\[
\dfrac{z}{e^z-1}=\sum_{j=0}^{\infty}\dfrac{B_j}{j!}z^j
\qquad(|z|<2\pi).
\]
For an integer $n\ge0$, let
\begin{equation}\label{eq:Rn}
\begin{split}
R_n(x)=(-1)^n\biggl[
&\log\Gamma(x)-\left(x-\dfrac12\right)\log x+x
-\dfrac12\log(2\pi)\\
&-\sum_{k=1}^{n}\dfrac{B_{2k}}{2k(2k-1)x^{2k-1}}
\biggr],\qquad x>0,
\end{split}
\end{equation}
where an empty sum is zero.
Throughout the main results, we use the notation
\begin{equation}\label{eq:notation}
f_{n,m}=(-1)^mR_n^{(m)},\qquad
p=2n+m-1,\qquad d=p-1,
\qquad n\ge2,\quad m\in\mathbb Z_{\ge0}.
\end{equation}
The classical Laplace representation recalled below shows that each $f_{n,m}$ is completely monotonic and tends to zero at infinity.

\subsection{Conjectures and prior results}

Complete monotonicity of the signed Stirling remainders was established by Alzer \cite[Theorem~8]{Alzer1997}.
Koumandos \cite{Koumandos2006} used Euler--Maclaurin bounds for the Bernoulli generating function to give simpler proofs of complete monotonicity results for gamma and polygamma expressions.
Diamond and Straub \cite{DiamondStraub2016} extended this analysis to shifted asymptotic expansions with Bernoulli-polynomial coefficients.
Earlier classifications of real power weights preserving complete monotonicity for differences of shifted polygamma functions appear in Alzer and Berg \cite[Theorem~4.3]{AlzerBerg2006}.

Koumandos and Pedersen \cite{KP2009} studied complete monotonicity of positive order and the Laplace kernels of Stirling remainders.
Their Remark~1.6 also gives a fractional-order formulation.
They later related the Laplace kernels of Euler and Barnes gamma remainders to absolutely monotonic functions \cite{KP2011}.
These structural results provide the setting for the degree questions considered below.

Qi and Liu \cite{QL2019} proposed conjectures for the completely monotonic degrees of gamma and polygamma remainders.
We recall separately the three real-degree predictions in the formulation of Mahmoud and Qi \cite{MQ2022}, whose Definition~1.1 explicitly permits real degrees.
The labels below are local to this paper, and the statements are restricted to the truncation range $n\ge2$.
Conjectures~\ref{conj:remainder}, \ref{conj:first}, and~\ref{conj:higher} correspond, respectively, to \cite[(2.5), (2.7), (2.9)]{MQ2022}.

\Needspace{6\baselineskip}
\begin{conjecture}\label{conj:remainder}
For every integer $n\ge2$,
\[
\dcm(R_n)=2n-2.
\]
\end{conjecture}

\Needspace{6\baselineskip}
\begin{conjecture}\label{conj:first}
For every integer $n\ge2$,
\[
\dcm(-R_n')=2n-1.
\]
\end{conjecture}

\Needspace{6\baselineskip}
\begin{conjecture}\label{conj:higher}
For all integers $n\ge2$ and $m\ge2$,
\[
\dcm\bigl((-1)^mR_n^{(m)}\bigr)=2n+m-2.
\]
\end{conjecture}

Together, Conjectures~\ref{conj:remainder}--\ref{conj:higher} assert
\begin{equation}\label{eq:conjecture}
\dcm(f_{n,m})=2n+m-2,
\qquad n\ge2,\quad m\ge0.
\end{equation}

Zhu \cite{Zhu2020} distinguished integer degrees and proved several low-order cases.
The remaining predictions for $R_n$ and $-R_n'$ are stated in \cite[(6.1), (6.2)]{Zhu2020} for $n\ge4$, while \cite[(6.3)]{Zhu2020} includes the $n\ge2$ branch for derivative orders $m\ge1$.
The following unified formulation combines these conjectures with the low-order cases already established there.

\Needspace{7\baselineskip}
\begin{conjecture}\label{conj:integer}
For all integers $n\ge2$ and $m\ge0$, the largest admissible integer exponent is
\[
\max\{j\in\mathbb Z_{\ge0}:x^j(-1)^mR_n^{(m)}(x)\in\CM\}
=2n+m-2.
\]
\end{conjecture}

Conjecture~\ref{conj:integer} concerns integer exponents only; it does not exclude admissible real exponents between consecutive integers.
The cases $n=0,1$ in the cited papers are outside the scope of the present results.
Xu and Cen \cite{XC2020} determined exact degrees in several of those cases and corrected some earlier predictions.
The present results concern the distinct range $n\ge2$.

\subsection{Main results and resolution of the conjectures}

Mahmoud and Qi \cite[Theorem~2.1]{MQ2022} proved
$2n-1\le\dcm(-R_n')<2n$ for $n\ge2$.
Bouali \cite[Proposition~1.1]{Bouali2022} states the companion bound
$2n-2\le\dcm(R_n)<2n-1$ in a preprint.
The Mahmoud--Qi interval already gives $\lfloor\dcm(-R_n')\rfloor=2n-1$ for every $n\ge2$, and the interval stated by Bouali gives $\lfloor\dcm(R_n)\rfloor=2n-2$.
Zhu \cite[Theorems~3.1--3.2 and 4.2--4.3]{Zhu2020} had established the corresponding integer-degree cases with $n=2,3$ and $m=0,1$.
Thus the integer formula in Conjecture~\ref{conj:integer} includes cases already available in the literature.

For the trigamma remainder, Qi \cite{Qi2021} proved $4\le\dcm(R_2'')\le5$.
None of these intervals alone determines whether a positive fractional improvement of the lower integer endpoint is admissible.
The theorem below gives a common proof of the integer-degree formula for every derivative order and establishes the strict fractional improvement throughout the family, including the cases whose integer degree was already known.

The quantitative bounds use the increasing function
\begin{equation}\label{eq:Aintro}
A(\alpha)=\sum_{k=1}^{\infty}\dfrac{\alpha}{k(k-\alpha)},
\qquad 0<\alpha<1.
\end{equation}
Lemma~\ref{lem:A} shows that $A$ increases continuously from zero to infinity.
For every integer $p\ge3$, define
\begin{equation}\label{eq:betauniversal}
\lambda_p=\dfrac12\min\left\{\tan\dfrac{\pi}{p+1},\dfrac1{\sqrt p}\right\},
\qquad
\beta_p=A^{-1}(\lambda_p).
\end{equation}

\begin{theorem}\label{thm:general}
For all integers $n\ge2$ and $m\ge0$,
\begin{equation}\label{eq:general}
d+\dfrac1{2p+1}<d+\beta_p
\le\dcm(f_{n,m})<d+1,
\qquad p=2n+m-1,\quad d=p-1.
\end{equation}
In particular, every equality in \eqref{eq:conjecture} fails, and
\begin{equation}\label{eq:integer}
\left\lfloor\dcm(f_{n,m})\right\rfloor
=\max\{j\in\mathbb Z_{\ge0}:x^j f_{n,m}(x)\in\CM\}
=2n+m-2.
\end{equation}
\end{theorem}

\noindent\textbf{Status of the conjectures.}
Theorem~\ref{thm:general} disproves Conjectures~\ref{conj:remainder}--\ref{conj:higher} in every case covered by their stated ranges and gives a unified proof of Conjecture~\ref{conj:integer}, incorporating the earlier integer-degree results described above.
Thus $2n+m-2$ is the exact integer part of the real degree, while the real degree itself is strictly larger.
For example, Theorem~\ref{thm:general} gives
\[
4+\dfrac1{11}<\dcm(R_2'')<5.
\]
For the first two degrees we can improve the general kernel estimates.
Put
\begin{equation}\label{eq:firstconstants}
M_H=\dfrac{96}{107+51\sqrt{17}},
\qquad
b_K=1-\cos^5\left(\dfrac{2\pi}{5}\right).
\end{equation}

\begin{theorem}\label{thm:first}
Let $\beta_H=A^{-1}(1/M_H)$ and $\beta_K=A^{-1}(9b_K/5)$.
Then
\begin{equation}\label{eq:firstbounds}
2+\beta_H\le\dcm(R_2)<3,
\qquad
3+\beta_K\le\dcm(-R_2')<4.
\end{equation}
In particular,
\begin{equation}\label{eq:rationalfirst}
\dfrac83<\dcm(R_2)<3,
\qquad
\dfrac{32}{9}<\dcm(-R_2')<4.
\end{equation}
\end{theorem}

For numerical orientation, the constants in Theorem~\ref{thm:first} satisfy
\[
\beta_H\approx0.728292486,
\qquad
\beta_K\approx0.572408378.
\]
These decimal approximations come from numerical solution of the defining equations $A(\beta_H)=1/M_H$ and $A(\beta_K)=9b_K/5$.
The rigorous lower bounds are the exact quantities $2+\beta_H$ and $3+\beta_K$ in \eqref{eq:firstbounds}; the numerical evaluations are not used in their proofs.
The inverse-function bounds are sufficient positivity certificates, and the exact degrees remain undetermined.
The fractional excess also satisfies a uniform asymptotic restriction.

\Needspace{12\baselineskip}
\begin{theorem}\label{thm:asymptotic}
Put
\[
\delta_{n,m}=\dcm(f_{n,m})-(2n+m-2),
\qquad
J=\int_0^{2\pi}\dfrac{1-\cos u}{u}\,du>0.
\]
Then
\begin{equation}\label{eq:asymptotic}
\limsup_{n\to\infty}(2n-1)
\sup_{m\in\mathbb Z_{\ge0}}\delta_{n,m}
\le\dfrac{2\pi^2}{J}.
\end{equation}
For each fixed $m\ge0$, one also has
\begin{equation}\label{eq:asymptoticlower}
\liminf_{n\to\infty}(2n-1)\delta_{n,m}\ge\dfrac3\pi.
\end{equation}
Consequently,
\begin{equation}\label{eq:exactscale}
\delta_{n,m}=\Theta(n^{-1})\quad\hbox{for each fixed }m,
\qquad
\sup_{m\ge0}\delta_{n,m}=\Theta(n^{-1}).
\end{equation}
\end{theorem}

Here $a_n=\Theta(n^{-1})$ means that $c/n\le a_n\le C/n$ for all sufficiently large $n$, for some constants $0<c\le C<\infty$.
The theorem determines the decay scale, but does not assert a limiting leading coefficient.
The upper bound is uniform over all derivative orders; the lower bound for an individual degree depends on their combined parameter $p=2n+m-1$.

\subsection{Contribution and proof strategy}

The rational Laplace representation, Bernstein's theorem, and the fractional-order interpretation of complete monotonicity are established tools; see \cite{KP2009}.
The signed-density characterization below is an application of classical uniqueness of Laplace transforms to this particular family.
The broader connection between derivative inequalities and positive integral representations includes the generalized Stieltjes characterizations of Sokal \cite{Sokal2010} and Koumandos and Pedersen \cite{KP2019}.
For generalized Stieltjes functions, monotonicity of asymptotic remainders can also characterize moment properties of the representing measures \cite{KP2015}.
Karp and Prilepkina \cite{KarpPrilepkina2012} developed criteria for exact Stieltjes order and fractional integral transformations of representing measures.
That order concerns generalized Stieltjes representations and is distinct from the weighted completely monotonic degree in \eqref{eq:degree}.

The quantitative step developed here is a fractional positivity criterion that uses the negative part of the kernel's derivative with respect to its logarithmic variable, only up to the point being tested.
For the common kernel, this local estimate yields an admissible fractional order comparable to $p^{-1}$.
It gives the simultaneous disproof of Conjectures~\ref{conj:remainder}--\ref{conj:higher}, extends the integer-degree formula to arbitrary derivative order, and supplies the lower half of the exact decay scale.
A separate scaled-density argument supplies the upper half, uniformly in all derivative orders.

Section~2 develops the positivity criteria and the elementary facts about admissible exponents.
Section~3 derives the common kernel and proves Theorem~\ref{thm:general}.
Section~4 sharpens the first two cases, and Section~5 establishes the asymptotic bounds, including derivative orders growing proportionally to the truncation order.
Section~6 states the remaining questions.

\section{Admissible exponents and fractional positivity}

\subsection{Admissible exponents}

\begin{lemma}\label{lem:closed}
If $f\in\CM$, the set $E_f=\{\rho\ge0:x^\rho f(x)\in\CM\}$ is downward closed and closed in $[0,\infty)$.
If $p\ge0$ and $p\notin E_f$, then $\sup E_f<p$.
\end{lemma}

\begin{proof}
If $0\le\sigma\le\rho\in E_f$, then
$x^\sigma f=x^{-(\rho-\sigma)}x^\rho f$ is completely monotonic by Leibniz's formula.
For each fixed $x>0$ and $j\ge0$, the quantity
$(-1)^j(d/dx)^j[x^\rho f(x)]$ depends continuously on $\rho$.
Each defining inequality is therefore closed in $\rho$, proving closedness of $E_f$.
Downward closure excludes all exponents at least $p$ when $p\notin E_f$; closedness excludes the possibility that its supremum equals $p$.
\end{proof}

\subsection{Fractional identities}

For a locally integrable function $h$ of at most polynomial growth, write
\[
\Lap h(x)=\int_0^\infty e^{-xt}h(t)\dd,
\qquad x>0.
\]
If $h\in C^1[0,\infty)$ and $h(0)=0$, define
\begin{equation}\label{eq:fractional}
\mathcal D^\alpha h(t)=\dfrac{1}{\Gamma(1-\alpha)}
\int_0^t(t-u)^{-\alpha}h'(u)\,du,
\qquad 0<\alpha<1.
\end{equation}
Under the condition $h(0)=0$, this is both the Caputo and the Riemann--Liouville derivative.

\begin{lemma}\label{lem:fractional}
Suppose that $h\in C^1[0,\infty)$, $h(0)=0$, and $h,h'$ have at most polynomial growth.
Then
\begin{align}
\Gamma(1-\alpha)t^\alpha\mathcal D^\alpha h(t)
&=h(t)+\alpha\int_0^1\dfrac{h(t)-h(tv)}{(1-v)^{1+\alpha}}\,dv,
\label{eq:Marchaud}\\
\Lap(\mathcal D^\alpha h)(x)&=x^\alpha\Lap h(x),
\label{eq:Lapfractional}\\
\mathcal D^\alpha[h(t/a)]
&=a^{-\alpha}(\mathcal D^\alpha h)(t/a),\qquad a>0.
\label{eq:scaling}
\end{align}
\end{lemma}

\begin{proof}
Substitute $u=tv$ in \eqref{eq:fractional} and integrate by parts using $h(t)-h(tv)$.
The boundary term at $v=1$ vanishes because $h(t)-h(tv)=O(1-v)$, and the term at $v=0$ equals $h(t)$.
This proves \eqref{eq:Marchaud}.

Absolute convergence allows Fubini's theorem, since
\[
\dfrac{1}{\Gamma(1-\alpha)}\int_0^\infty e^{-xt}
\int_0^t(t-u)^{-\alpha}|h'(u)|\,du\dd
=x^{\alpha-1}\int_0^\infty e^{-xu}|h'(u)|\,du<\infty.
\]
Thus $\Lap(\mathcal D^\alpha h)=x^{\alpha-1}\Lap(h')=x^\alpha\Lap h$, where the last step is integration by parts.
A change of variable in \eqref{eq:fractional} proves \eqref{eq:scaling}.
\end{proof}

\begin{lemma}\label{lem:A}
The function in \eqref{eq:Aintro} has the equivalent representations
\begin{equation}\label{eq:A}
A(\alpha)=\alpha\int_0^1\dfrac{-\log v}{(1-v)^{1+\alpha}}\,dv
=\int_0^1\dfrac{t^{-\alpha}-1}{1-t}\,dt.
\end{equation}
It is continuous and strictly increasing on $(0,1)$, with limits zero and infinity at the respective endpoints.
Moreover,
\begin{equation}\label{eq:Abound}
A(\alpha)<\dfrac{2\alpha}{1-\alpha},
\qquad
A(1/2)=2\log2,
\qquad
A(2/3)=\dfrac32\log3+\dfrac{\pi}{2\sqrt3}.
\end{equation}
\end{lemma}

\begin{proof}
Both integrals in \eqref{eq:A} converge.
Integrating the first by parts against $(1-v)^{-\alpha}-1$, whose boundary terms vanish, and then putting $t=1-v$ gives the second.
Expanding $(1-t)^{-1}$ and using Tonelli's theorem yields
\[
A(\alpha)=\sum_{k=1}^\infty
\left(\dfrac{1}{k-\alpha}-\dfrac1k\right),
\]
which proves \eqref{eq:Aintro}.
The series converges locally uniformly in $\alpha$; each summand increases strictly, and the first tends to infinity as $\alpha\uparrow1$.
Furthermore,
\[
A(\alpha)\le\dfrac{\alpha}{1-\alpha}\sum_{k=1}^\infty k^{-2}
<\dfrac{2\alpha}{1-\alpha},
\]
using $\sum_{k=1}^\infty k^{-2}<1+\int_1^\infty u^{-2}\,du=2$.
This also gives the limit at zero.
The substitution $t=u^2$ proves $A(1/2)=2\log2$.
Similarly, $t=u^3$ gives
\[
A(2/3)=3\int_0^1\dfrac{1+u}{1+u+u^2}\,du
=\dfrac32\log3+\dfrac{\pi}{2\sqrt3}.
\]
\end{proof}

\subsection{Local and global slope criteria}

The quantity $u h'(u)$ is the derivative of $h$ with respect to the logarithmic variable:
\[
u h'(u)=\left.\dfrac{d}{dv}h(e^v)\right|_{v=\log u},\qquad u>0.
\]
The slope estimates below control its negative part, without division by $h(u)$.
Thus the relevant quantity is $-u h'(u)$, rather than the normalized logarithmic derivative $-u h'(u)/h(u)$.

\begin{proposition}\label{prop:localpositivity}
Under the hypotheses of Lemma~\ref{lem:fractional}, define
\[
M_h(t)=\max_{0\le u\le t}\bigl(-u h'(u)\bigr)_+,
\qquad a_+=\max\{a,0\}.
\]
Then, for $t>0$ and $0<\alpha<1$,
\begin{equation}\label{eq:localcriterion}
\Gamma(1-\alpha)t^\alpha\mathcal D^\alpha h(t)
\ge h(t)-A(\alpha)M_h(t).
\end{equation}
In particular, if $h'\ge0$ on $(0,T]$ and
$h(t)\ge\lambda M_h(t)$ for $t\ge T$, where $T,\lambda>0$, then
$\mathcal D^\alpha h\ge0$ and $x^\alpha\Lap h\in\CM$ whenever $A(\alpha)\le\lambda$.
\end{proposition}

\begin{proof}
For $0<v<1$, the definition of $M_h(t)$ gives
\[
h(t)-h(tv)=\int_{tv}^t h'(u)\,du
\ge-M_h(t)\int_{tv}^t\dfrac{du}{u}=M_h(t)\log v.
\]
Substitution into \eqref{eq:Marchaud} proves \eqref{eq:localcriterion}.
On $(0,T]$, nonnegativity follows directly from \eqref{eq:fractional}.
For $t\ge T$, it follows from \eqref{eq:localcriterion} and $A(\alpha)\le\lambda$.
Finally, \eqref{eq:Lapfractional} gives complete monotonicity by differentiation under the integral.
\end{proof}

\begin{corollary}\label{prop:positivity}
Under the hypotheses of Lemma~\ref{lem:fractional}, suppose that $T,b>0$ and $M\ge0$ satisfy
\begin{equation}\label{eq:criterion}
\begin{gathered}
h'(t)\ge0\quad(0<t\le T),\qquad
h(t)\ge b\quad(t\ge T),\\
t h'(t)\ge-M\quad(t>0).
\end{gathered}
\end{equation}
If $M A(\alpha)\le b$, then $\mathcal D^\alpha h\ge0$ and $x^\alpha\Lap h\in\CM$.
If $MA(\alpha_0)<b$, the same conclusion holds for every $\alpha$ in a sufficiently small neighborhood of $\alpha_0$ in $(0,1)$.
\end{corollary}

\begin{proof}
The global slope assumption gives $M_h(t)\le M$.
For $t\ge T$, Proposition~\ref{prop:localpositivity} yields
\begin{equation}\label{eq:lowerdensity}
\Gamma(1-\alpha)t^\alpha\mathcal D^\alpha h(t)
\ge b-MA(\alpha)\ge0.
\end{equation}
For $t\le T$, use $h'\ge0$ in \eqref{eq:fractional}.
Continuity of $A$ preserves a strict margin near $\alpha_0$.
\end{proof}

Only the negative part of $th'(t)$ enters the criterion.
This is useful when a kernel initially increases, then has small oscillations around a positive limit.

\section{The common kernel and noninteger degrees}

\subsection{The classical remainder representation}

For an integer $p\ge1$, put
\begin{equation}\label{eq:hp}
h_p(s)=1-\operatorname{Re}(1+is)^{-p},\qquad s\ge0.
\end{equation}
The function is smooth and bounded, vanishes at zero, and is strictly positive for $s>0$, since $|(1+is)^{-p}|<1$.
It also satisfies $h_p'(s)=O(s)$ at zero.

\begin{lemma}\label{lem:kernel}
Let $a_k=2\pi k$.
For every $n\ge1$,
\begin{equation}\label{eq:rn}
R_n=\Lap r_n,
\qquad
r_n(t)=2\sum_{k=1}^\infty\dfrac{t^{2n}}{a_k^{2n}(t^2+a_k^2)}.
\end{equation}
For all integers $m\ge0$, with $p=2n+m-1$ and $d=p-1$,
\begin{align}
(t^m r_n(t))^{(d)}
&=2d!\sum_{k=1}^\infty a_k^{-2n}h_p(t/a_k),\label{eq:kernelderiv}\\
x^d f_{n,m}(x)
&=2d!\sum_{k=1}^\infty a_k^{-2n}\Lap[h_p(t/a_k)](x).
\label{eq:kernelLap}
\end{align}
\end{lemma}

\begin{proof}
Binet's formula and the partial-fraction expansion of $\coth$ give
\begin{equation}\label{eq:Binet}
R_0=\Lap r_0,
\qquad
r_0(t)=\dfrac1t\left(\dfrac1{e^t-1}-\dfrac1t+\dfrac12\right)
=2\sum_{k=1}^\infty\dfrac1{t^2+a_k^2};
\end{equation}
see \cite[Section~2, (3)--(4) and Theorem~2.1]{KP2009}.
The finite geometric expansion
\[
\dfrac1{t^2+a^2}
=\sum_{j=0}^{n-1}\dfrac{(-1)^j t^{2j}}{a^{2j+2}}
+\dfrac{(-1)^n t^{2n}}{a^{2n}(t^2+a^2)}
\]
proves \eqref{eq:rn} after Laplace transformation.
Indeed, comparison of \eqref{eq:Binet} with the Bernoulli generating function identifies the polynomial coefficients as $B_{2j+2}/(2j+2)!$.
Their Laplace transforms are precisely the terms subtracted in \eqref{eq:Rn}.
The same remainder representation appears in \cite[(3)--(4)]{KP2009}.

Set $F_p(s)=s^{p+1}/(1+s^2)$.
We claim that
\begin{equation}\label{eq:algebraic}
F_p^{(p-1)}(s)=(p-1)!h_p(s).
\end{equation}
For $p=1$, this follows directly from \eqref{eq:hp}.
The identities $F_{p+1}=sF_p$ and
\[
h_{p+1}(s)=h_p(s)+\dfrac{s}{p}h_p'(s)
\]
then prove \eqref{eq:algebraic} by induction and Leibniz's formula.
Since
\[
\dfrac{2t^{2n+m}}{a_k^{2n}(t^2+a_k^2)}
=2a_k^{m-2}F_p(t/a_k),
\]
$d$ differentiations give \eqref{eq:kernelderiv}.
All these differentiated series converge uniformly on compact intervals: each derivative of fixed order is bounded there by a constant times $a_k^{-2n-2}$.

Write $g(t)=t^m r_n(t)$.
The series at zero gives $g(t)=O(t^{d+2})$ and $g^{(j)}(0)=0$ for $j\le d$.
The right side of \eqref{eq:kernelderiv} is bounded globally because $|h_p|\le2$ and $\sum a_k^{-2n}<\infty$.
Repeated integration therefore gives $g^{(j)}(t)=O((1+t)^{d-j})$ for $0\le j\le d$.
These estimates justify differentiating $R_n=\Lap r_n$ to obtain $f_{n,m}=\Lap g$, and all boundary terms in $d$ integrations by parts vanish.
This proves \eqref{eq:kernelLap}; summation under the integral is justified by boundedness or Tonelli's theorem.
\end{proof}

\subsection{The sign of the Laplace density}

For later use, we need a density test valid even when the density changes sign.

\begin{proposition}\label{prop:density}
For $n\ge2$, $m\ge0$, and $0<\alpha<1$, define
\begin{equation}\label{eq:Q}
Q_{n,m,\alpha}(t)=2d!\sum_{k=1}^\infty
 a_k^{-2n-\alpha}(\mathcal D^\alpha h_p)(t/a_k).
\end{equation}
Then $Q_{n,m,\alpha}$ is continuous, has at most polynomial growth, and
\begin{equation}\label{eq:QLap}
x^{d+\alpha}f_{n,m}(x)=\Lap Q_{n,m,\alpha}(x).
\end{equation}
Moreover,
\begin{equation}\label{eq:iff}
x^{d+\alpha}f_{n,m}(x)\in\CM
\quad\Longleftrightarrow\quad
Q_{n,m,\alpha}(t)\ge0\ \hbox{for every }t>0.
\end{equation}
\end{proposition}

\begin{proof}
Differentiation of \eqref{eq:hp} yields
\begin{equation}\label{eq:hpprime}
h_p'(s)=p(1+s^2)^{-(p+1)/2}
\sin\bigl((p+1)\arctan s\bigr),
\end{equation}
so $|h_p'|\le p$ and
\begin{equation}\label{eq:derivbound}
|\mathcal D^\alpha h_p(s)|\le\dfrac{p s^{1-\alpha}}{\Gamma(2-\alpha)}.
\end{equation}
Consequently, the series in \eqref{eq:Q} converges locally uniformly and is bounded in absolute value by a constant times $t^{1-\alpha}$, using $\sum a_k^{-2n-1}<\infty$.
Equations~\eqref{eq:Lapfractional}, \eqref{eq:scaling}, and \eqref{eq:kernelLap} now prove \eqref{eq:QLap}, with absolutely convergent interchanges.

A nonnegative density gives complete monotonicity.
Conversely, Bernstein's theorem represents a completely monotonic function as the Laplace transform of a nonnegative measure $\mu$ on $[0,\infty)$; see \cite[Section~1]{KP2009} and \cite{Sokal2010}.
Fix $x_0>0$.
The finite signed measure
\[
e^{-x_0t}\bigl(\mu(dt)-Q_{n,m,\alpha}(t)\,dt\bigr)
\]
has zero Laplace transform.
To see uniqueness directly, push this measure forward under $y=e^{-t}$ and extend it to $[0,1]$ by assigning zero mass to $0$.
All its nonnegative integer moments vanish, including the zeroth moment.
Uniform polynomial approximation on $[0,1]$ then shows that this signed measure is zero.
Hence $\mu(dt)=Q_{n,m,\alpha}(t)\,dt$, so $Q_{n,m,\alpha}\ge0$ almost everywhere and, by continuity, everywhere on $(0,\infty)$.
\end{proof}

\subsection{A universal fractional improvement}

\begin{lemma}\label{lem:universal}
For every integer $p\ge3$, put
\[
T_p=\tan\dfrac{\pi}{p+1},
\qquad
\lambda_p=\dfrac12\min\{T_p,p^{-1/2}\},
\qquad
\beta_p=A^{-1}(\lambda_p).
\]
Then $\mathcal D^\alpha h_p\ge0$ for every $0<\alpha\le\beta_p$.
Moreover,
\[
\beta_p>\dfrac1{2p+1}
>\dfrac2{(p+1)^2+1},
\qquad
\lim_{p\to\infty}p\beta_p=\dfrac3\pi.
\]
\end{lemma}

\begin{proof}
Equation~\eqref{eq:hpprime} shows that $h_p'\ge0$ on $(0,T_p]$.
Set
\[
q_p(u)=pu(1+u^2)^{-(p+1)/2},
\qquad u_p=p^{-1/2}.
\]
Then $|u h_p'(u)|\le q_p(u)$, while
\[
q_p'(u)=p(1+u^2)^{-(p+3)/2}(1-pu^2).
\]
Thus $q_p$ increases up to $u_p$ and decreases thereafter.
For $0<t\le u_p$, it follows that $M_{h_p}(t)\le q_p(t)$.
Also,
\[
h_p(t)\ge1-(1+t^2)^{-p/2}
=q_p(t)\dfrac{\sqrt{1+t^2}\bigl((1+t^2)^{p/2}-1\bigr)}{pt}
\ge\dfrac t2q_p(t),
\]
where the last step is Bernoulli's inequality, since $p/2\ge1$.
Hence $h_p(t)\ge(t/2)M_{h_p}(t)$ in this range.
For $t\ge u_p$, the amplitude bound and the preceding estimate at $u_p$ give
\[
h_p(t)\ge1-(1+u_p^2)^{-p/2}
\ge\dfrac{u_p}{2}q_p(u_p)
\ge\dfrac{u_p}{2}M_{h_p}(t).
\]
Combining the two ranges shows that
$h_p(t)\ge\lambda_pM_{h_p}(t)$ whenever $t\ge T_p$.
The pointwise criterion~\eqref{eq:localcriterion} proves the positivity assertion.

For $p\ge4$, one has $T_p>\pi/(p+1)>2/p$ and $p^{-1/2}/2\ge1/p$, so $\lambda_p\ge1/p$.
By~\eqref{eq:Abound},
\[
A\left(\dfrac1{2p+1}\right)
<\dfrac{2/(2p+1)}{1-1/(2p+1)}
=\dfrac1p\le\lambda_p.
\]
For $p=3$, $\lambda_3=1/(2\sqrt3)$, and the defining series yields
\[
A(1/7)\le\dfrac16\sum_{k=1}^\infty k^{-2}
=\dfrac{\pi^2}{36}<\dfrac1{2\sqrt3},
\]
using $\pi^2<10<6\sqrt3$.
Strict monotonicity of $A$ proves $\beta_p>1/(2p+1)$.
The second strict inequality is equivalent to $p(p-2)>0$.

Finally, $p\lambda_p\to\pi/2$ and
\[
\lim_{\alpha\downarrow0}\dfrac{A(\alpha)}\alpha
=\sum_{k=1}^\infty k^{-2}=\dfrac{\pi^2}{6},
\]
by dominated convergence in the defining series.
Since $\lambda_p\to0$, inversion gives $p\beta_p\to3/\pi$.
\end{proof}

\subsection{The upper endpoint and proof of the main theorem}

\begin{proposition}\label{prop:upper}
For every $n\ge2$ and $m\ge0$, the function $x^p f_{n,m}(x)$ is not convex near zero.
In particular, $\dcm(f_{n,m})<p$.
\end{proposition}

\begin{proof}
Put $c_k=|B_{2k}|/[2k(2k-1)]>0$, and use the rising factorial
$(a)_m=a(a+1)\cdots(a+m-1)$ with $(a)_0=1$.
The identity $\Gamma(1+x)=x\Gamma(x)$ gives
\[
C(x):=R_0(x)=-\dfrac12\log x-x\log x+B(x),
\qquad
B(x)=\log\Gamma(1+x)+x-\dfrac12\log(2\pi),
\]
where $B$ is analytic near zero.
The Bernoulli signs then give the exact identity
\begin{equation}\label{eq:origin}
x^p f_{n,m}(x)
=\sum_{k=1}^n(-1)^{n+k}c_k(2k-1)_m x^{2(n-k)}
+(-1)^{n+m}x^p C^{(m)}(x).
\end{equation}
For every fixed $r\ge0$, the explicit formula for $C$ gives
$C^{(r)}(x)=O(x^{-r}(1+|\log x|))$ as $x\downarrow0$.
Differentiating the last term of \eqref{eq:origin} twice by Leibniz's formula, rather than differentiating an error estimate, bounds each contribution by
$O(x^{2n-3}(1+|\log x|))$, which tends to zero.
Only $k=n-1$ in the finite sum contributes to the limiting second derivative.
Thus
\begin{equation}\label{eq:nonconvex}
\lim_{x\downarrow0}(x^p f_{n,m}(x))''
=-2c_{n-1}(2n-3)_m<0.
\end{equation}
Complete monotonicity implies convexity, so $p$ is not admissible.
Lemma~\ref{lem:closed} proves the strict degree bound.
\end{proof}

\begin{proof}[Proof of Theorem~\ref{thm:general}]
Lemma~\ref{lem:universal} gives $\mathcal D^{\beta_p}h_p\ge0$ and $\beta_p>1/(2p+1)$.
Every summand of $Q_{n,m,\beta_p}$ is therefore nonnegative.
Proposition~\ref{prop:density} yields
\[
\dcm(f_{n,m})\ge d+\beta_p>d+\dfrac1{2p+1}.
\]
Proposition~\ref{prop:upper} proves the upper inequality.
Since $d$ and $d+1$ are consecutive integers, the largest admissible integer is $d=2n+m-2$.
\end{proof}

\section{Stronger bounds for the first two degrees}

For reference, the two functions in Theorem~\ref{thm:first} are
\begin{align}
R_2(x)
&=\log\Gamma(x)-\left(x-\dfrac12\right)\log x+x
-\dfrac12\log(2\pi)-\dfrac1{12x}+\dfrac1{360x^3},\label{eq:R2}\\
-R_2'(x)
&=\log x-\psi(x)-\dfrac1{2x}-\dfrac1{12x^2}+\dfrac1{120x^4},
\label{eq:dR2}
\end{align}
where $\psi=\Gamma'/\Gamma$.
The associated kernels are
\begin{align}
H(s):=h_3(s)&=1+\dfrac{3s^2-1}{(1+s^2)^3},\label{eq:H}\\
K(s):=h_4(s)&=1-\dfrac{s^4-6s^2+1}{(1+s^2)^4}.
\label{eq:K}
\end{align}
Both are positive for $s>0$ and satisfy the regularity hypotheses of Corollary~\ref{prop:positivity}.

\begin{lemma}\label{lem:Hbounds}
The function $H$ satisfies \eqref{eq:criterion} with
$(T,b,M)=(1,1,M_H)$.
Moreover, $M_H$ is the least possible constant in the global estimate
$sH'(s)\ge-M_H$.
\end{lemma}

\begin{proof}
Direct differentiation gives
\[
H'(s)=\dfrac{12s(1-s^2)}{(1+s^2)^4}.
\]
Thus $H'\ge0$ on $(0,1]$, while \eqref{eq:H} gives $H(s)\ge1$ for $s\ge1$.
For $y=s^2\ge1$, put
\[
q(y)=\dfrac{12y(y-1)}{(1+y)^4}.
\]
Then $-sH'(s)=q(y)$ and
\[
q'(y)=\dfrac{12(-2y^2+5y-1)}{(1+y)^5}.
\]
The only zero of $q'$ in $[1,\infty)$ is $r_H=(5+\sqrt{17})/4$.
The function $q$ increases up to $r_H$ and decreases thereafter, with
$q(1)=0$ and $q(y)\to0$ as $y\to\infty$.
Consequently its maximum is $q(r_H)=M_H$, attained when $s=\sqrt{r_H}$.
For $s\le1$, $sH'(s)$ is nonnegative.
\end{proof}

\begin{lemma}\label{lem:Kbounds}
The function $K$ satisfies \eqref{eq:criterion} with
\[
(T,b,M)=\left(\dfrac12,b_K,\dfrac59\right).
\]
\end{lemma}

\begin{proof}
Put $y=s^2$.
Then
\begin{equation}\label{eq:Kprime}
K'(s)=\dfrac{4s(y^2-10y+5)}{(1+y)^5}.
\end{equation}
For $0<s\le1/2$, the polynomial in the numerator is at least $5-10/4>0$.
Put $\theta=\arctan s$.
The derivative also equals $4(1+s^2)^{-5/2}\sin(5\theta)$, so the positive critical points are
$s_1=\tan(\pi/5)$ and $s_2=\tan(2\pi/5)$.
The function increases up to $s_1$, decreases from $s_1$ to $s_2$, and then increases to its limit $1$.
Since $s_1>1/2$ and $K(1/2)>1$, its minimum on $[1/2,\infty)$ is
\[
K(s_2)=1-\cos^4(2\pi/5)\cos(8\pi/5)
=1-\cos^5(2\pi/5)=b_K.
\]

For the slope bound on $0\le y\le1$, use
\begin{align*}
&(1+y)^5+8y(y^2-10y+5)\\
&\qquad=(1-y)(1+46y-24y^2-6y^3-y^4)\\
&\qquad\ge(1-y)(1+15y)\ge0.
\end{align*}
Together with \eqref{eq:Kprime}, this gives $sK'(s)\ge-1/2$.
For $y\ge1$, set $z=y-1\ge0$ and use the exact identity
\begin{align}
&5(1+y)^5+36y(y^2-10y+5)\nonumber\\
&\qquad=5z^5+50z^4+236z^3+132z^2+16(z-1)^2\ge0.
\label{eq:Kcertificate}
\end{align}
It follows that $sK'(s)\ge-5/9$ on this range.
The global one-sided estimate is therefore $sK'(s)\ge-5/9$.
\end{proof}

\begin{proof}[Proof of Theorem~\ref{thm:first}]
By Lemma~\ref{lem:A}, there are unique $\beta_H,\beta_K\in(0,1)$ with
$A(\beta_H)=1/M_H$ and $A(\beta_K)=9b_K/5$.
The preceding two lemmas and Corollary~\ref{prop:positivity} give
\[
\mathcal D^\alpha H\ge0\quad(0<\alpha\le\beta_H),
\qquad
\mathcal D^\alpha K\ge0\quad(0<\alpha\le\beta_K).
\]
In these respective ranges, Proposition~\ref{prop:density} gives the nonnegative densities
\begin{align}
Q_{2,0,\alpha}(t)
&=4\sum_{k=1}^\infty a_k^{-4-\alpha}(\mathcal D^\alpha H)(t/a_k),
\label{eq:QH}\\
Q_{2,1,\alpha}(t)
&=12\sum_{k=1}^\infty a_k^{-4-\alpha}(\mathcal D^\alpha K)(t/a_k).
\label{eq:QK}
\end{align}
Their Laplace transforms are $x^{2+\alpha}R_2(x)$ and $x^{3+\alpha}[-R_2'(x)]$, respectively.
Taking $\alpha=\beta_H$ and $\alpha=\beta_K$ proves the lower bounds in \eqref{eq:firstbounds}.
The upper bounds follow from Theorem~\ref{thm:general}.

For the explicit rational bounds, the integral from Lemma~\ref{lem:A} gives
\[
A(2/3)=3\int_0^1\dfrac{1+u}{1+u+u^2}\,du<3,
\]
since the integrand is less than one for $u>0$.
Since $1/M_H=(107+51\sqrt{17})/96>3$, we obtain $\beta_H>2/3$.
For $k\ge2$, we have $k-5/9\ge13k/18$.
Consequently,
\[
A(5/9)\le\dfrac54+\dfrac{10}{13}
\left(\sum_{k=1}^\infty k^{-2}-1\right)
<\dfrac54+\dfrac{10}{13}\cdot\dfrac{13}{20}
=\dfrac74.
\]
Here $\sum k^{-2}=\pi^2/6<33/20$, for example from $\pi<22/7$.
Also, $\cos(2\pi/5)=(\sqrt5-1)/4<1/3$, so $b_K>1-1/243>35/36$.
Hence $9b_K/5>7/4>A(5/9)$, which proves $\beta_K>5/9$ and completes \eqref{eq:rationalfirst}.
\end{proof}

\section{The order of the fractional excess}

The upper endpoint obstruction in Proposition~\ref{prop:upper} is local in the original variable $x$.
A complementary obstruction comes from the fractional density at a shrinking positive point.
After rescaling, the kernels converge to $1-\cos u$, whose first return to zero forces a negative fractional derivative.
The following argument quantifies this obstruction uniformly in $m$.

\begin{proof}[Proof of Theorem~\ref{thm:asymptotic}]

Write $L=2\pi$, $N=2n-1$, $p=N+m$, and
\[
 g_p(u)=h_p(u/p),\qquad g(u)=1-\cos u.
\]
Directly from the derivative formula,
\[
 g_p'(u)=\left(1+\dfrac{u^2}{p^2}\right)^{-(p+1)/2}
             \sin\left((p+1)\arctan\dfrac{u}{p}\right),
 \qquad |g_p'(u)|\le1.
\]
Taylor expansion, uniformly on $[0,L]$, gives
\begin{equation}
 \|g_p-g\|_{C^1[0,L]}=O(p^{-1}),\qquad
 g_p(L)=\dfrac{L^2}{2p}+O(p^{-2}).                 \label{eq:C1limit}
\end{equation}
For example, the second formula follows from
$(1+iL/p)^{-p}=\exp(-iL-L^2/(2p)+O(p^{-2}))$ and $e^{-iL}=1$.

We first record the small-order expansion, uniform for large $p$ and
$0\le\alpha\le1/2$:
\begin{equation}
\mathcal D^\alpha g_p(L)
 =\dfrac{L^2}{2p}-\alpha J
       +O\left(p^{-2}+\dfrac{\alpha}{p}+\alpha^2\right).       \label{eq:smallalpha}
\end{equation}
Here $\mathcal D^0g_p=g_p$.
The implied constant in \eqref{eq:smallalpha} depends only on the fixed interval $[0,L]$, uniformly for large $p$ and $\alpha\in[0,1/2]$; in particular, it is independent of $m$.
Indeed, the fractional integration-by-parts formula is
\[
 \mathcal D^\alpha g_p(L)
 =\dfrac{L^{-\alpha}}{\Gamma(1-\alpha)}
  \left[g_p(L)+\alpha
   \int_0^1\dfrac{g_p(L)-g_p(Lv)}{(1-v)^{1+\alpha}}\,dv\right].
\]
Replacing $g_p$ by $g$ inside the integral incurs $O(p^{-1})$, uniformly
for $\alpha\le1/2$, since the difference of its two numerator values is
bounded by $L(1-v)\|(g_p-g)'\|_\infty$.
For $g$, replacing $\alpha$ by zero incurs $O(\alpha)$: use
$|g(L)-g(Lv)|\le L(1-v)$ and
\[
 |(1-v)^{-\alpha}-1|
 \le\alpha|\log(1-v)|(1-v)^{-1/2}.
\]
The integral at $\alpha=0$ is
\[
 -\int_0^1\dfrac{1-\cos(Lv)}{1-v}\,dv
 =-\int_0^L\dfrac{1-\cos u}{u}\,du=-J.
\]
Together with \eqref{eq:C1limit} and $L^{-\alpha}/\Gamma(1-\alpha)=1+O(\alpha)$,
these estimates prove \eqref{eq:smallalpha}.

Evaluate the density at $t_p=2\pi L/p=4\pi^2/p$.
Fractional scaling yields
\begin{equation}
\dfrac{(2\pi)^{2n+\alpha}}{2d!p^\alpha}
       Q_{n,m,\alpha}(t_p)
 =\sum_{k=1}^{\infty}k^{-2n-\alpha}
                         \mathcal D^\alpha g_p(L/k).                 \label{eq:scaledQ}
\end{equation}
The bound $|g_p'|\le1$ gives
\begin{equation}
\left|\sum_{k=2}^{\infty}k^{-2n-\alpha}\mathcal D^\alpha g_p(L/k)\right|
 \le\dfrac{L^{1-\alpha}}{\Gamma(2-\alpha)}
                \sum_{k=2}^{\infty}k^{-2n-1}
 =O(4^{-n}),                                                \label{eq:tail}
\end{equation}
uniformly in $p$ and $\alpha\in[0,1/2]$.
More explicitly, the integral test gives
\[
\sum_{k=2}^{\infty}k^{-2n-1}
\le2^{-2n-1}+\int_2^\infty u^{-2n-1}\,du
=4^{-n}\left(\dfrac12+\dfrac1{2n}\right),
\]
and $L^{1-\alpha}/\Gamma(2-\alpha)$ is bounded on $[0,1/2]$.

Fix $c>L^2/(2J)$ and put $\alpha_n=c/N$.
For large $n$, this lies in $(0,1/2]$.
Since $p\ge N$, \eqref{eq:smallalpha}--\eqref{eq:tail} give, uniformly in $m\ge0$,
\[
 \dfrac{(2\pi)^{2n+\alpha_n}}{2d!p^{\alpha_n}}
       Q_{n,m,\alpha_n}(t_p)
 \le \dfrac{L^2/2-cJ}{N}+O(N^{-2})+O(4^{-n})<0.
\]
The implied constants here depend only on $L$ and the fixed choice of $c$, not on $m$.
Proposition~\ref{prop:density} therefore excludes complete monotonicity at exponent
$d+\alpha_n$ for every $m\ge0$.
Lemma~\ref{lem:closed} gives $\delta_{n,m}<\alpha_n$, and hence
$N\sup_{m\ge0}\delta_{n,m}\le c$ for all sufficiently large $n$.
Letting $c\downarrow L^2/(2J)$ proves the asserted quantitative bound.
For fixed $m$, Lemma~\ref{lem:universal} gives
\[
N\beta_{N+m}
=\dfrac{N}{N+m}(N+m)\beta_{N+m}\longrightarrow\dfrac3\pi.
\]
Since $\delta_{n,m}\ge\beta_{N+m}$ by Theorem~\ref{thm:general}, this proves \eqref{eq:asymptoticlower}.
The upper estimate and this positive lower limit give the first assertion of \eqref{eq:exactscale}.
For the supremum, use $\sup_{m\ge0}\delta_{n,m}\ge\delta_{n,0}$ together with the uniform upper bound.
\end{proof}

\Needspace{15\baselineskip}
\begin{corollary}\label{cor:growingorders}
For every fixed $C\ge0$, with $N=2n-1$,
\begin{equation}\label{eq:growingorders}
\liminf_{n\to\infty}
N\inf_{\substack{m\in\mathbb Z_{\ge0}\\m\le CN}}\delta_{n,m}
\ge\dfrac{3}{\pi(1+C)}.
\end{equation}
Consequently, the bounds $c_C/n\le\delta_{n,m}\le C_0/n$ hold uniformly for $0\le m\le C(2n-1)$ and sufficiently large $n$, with positive constants $c_C,C_0$.
\end{corollary}

\begin{proof}
Given $0<\varepsilon<3/\pi$, Lemma~\ref{lem:universal} gives
$p\beta_p\ge3/\pi-\varepsilon$ for all sufficiently large $p$.
For $N\le p=N+m\le(1+C)N$, Theorem~\ref{thm:general} therefore gives
\[
N\delta_{n,m}\ge\dfrac{N}{p}\,p\beta_p
\ge\dfrac{3/\pi-\varepsilon}{1+C}.
\]
Take the infimum, then the lower limit, and let $\varepsilon\downarrow0$.
The upper estimate follows from Theorem~\ref{thm:asymptotic}.
\end{proof}

\section{Scope of the results and open questions}

Theorems~\ref{thm:general}--\ref{thm:asymptotic} resolve the distinction between the real and integer versions of the conjectures for every $n\ge2$ and $m\ge0$.
The real degrees are nonintegers, their integer parts are exactly $2n+m-2$, and the excess has order $n^{-1}$ for fixed derivative order.
The same order holds uniformly when the derivative order grows at most proportionally to $n$, as stated in Corollary~\ref{cor:growingorders}.

The positivity criteria give sufficient conditions for each individual scaled kernel to have a nonnegative fractional derivative.
The density characterization in Proposition~\ref{prop:density} instead tests the full sum $Q_{n,m,\alpha}$.
Cancellation between summands may preserve nonnegativity beyond the range certified term by term.
Thus the lower bounds, including those with exact auxiliary kernel constants in Section~4, are not asserted to be the exact degrees.
In particular, optimality of $M_H$ in the inequality $sH'(s)\ge-M_H$ does not imply $\dcm(R_2)=2+\beta_H$.

\Needspace{10\baselineskip}
The first remaining problem is to determine the exact values of $\dcm(R_2)$ and $\dcm(-R_2')$.
A second is to decide whether $(2n-1)\delta_{n,m}$ converges for fixed $m$, and to identify its limit if it exists.
The present estimates place its lower and upper limits between $3/\pi$ and $2\pi^2/J$.
The limit $p\beta_p\to3/\pi$ describes the constructed lower bound, and does not identify the limiting coefficient of the actual degree.
Determining the optimal constants in the uniform estimates is a separate question.
Finally, the behavior as $m\to\infty$ with $n$ fixed is not settled by the asymptotics in this paper.

\Needspace{8\baselineskip}

\end{document}